\documentclass[11pt]{amsart}
\usepackage[light,onlyrm,notext,nosf,sfmathbb,frenchstyle]{kpfonts}
\usepackage{graphicx,color}
\usepackage{amssymb,latexsym,hyperref}
\usepackage[mathscr]{euscript}

\usepackage{tikz}
\usepackage[backgroundcolor=yellow,%textsize=tiny,textwidth=.4\marginparwidth
]{todonotes}
\usepackage{tikz-cd}

\usepackage{mathrsfs}
\newcommand{\curly}[1]{\mathscr{#1}}

\newcommand\A{\mathcal A}
\newcommand\I{\curly I}

\makeatletter

\let\O\@undefined
\makeatother
\newcommand\O{\mathcal O}

\newcommand\PP{\mathbb P}
\makeatletter

\let\P\@undefined
\makeatother
\newcommand\P{\mathbb P}
\newcommand\C{\mathbb C}
\newcommand\Q{\mathbb Q}

\newcommand{\rt}[1]{\stackrel{#1\,}{\rightarrow}}
\newcommand{\Rt}[1]{\stackrel{#1\,}{\longrightarrow}}
\newcommand{\RT}[2]{\xymatrix@C=#1pt{\ar[r]^{#2}&}}
\newcommand\To{\longrightarrow}
\newcommand\into{\hookrightarrow}
\newcommand\ito{\ar@{^{ (}->}[r]}
\newcommand{\Into}{\ensuremath{\lhook\joinrel\relbar\joinrel\rightarrow}}

\newcommand{\onto}{\twoheadrightarrow}

\renewcommand\_{^{}_}

\newfont{\bigtimesfont}{cmsy10 scaled \magstep5}
\newcommand{\bigtimes}{\mathop{\lower0.9ex\hbox{\bigtimesfont\symbol2}}}
\newcommand\Langle{\big\langle}
\newcommand\Rangle{\big\rangle}

\newcommand\Bl{\operatorname{Bl}}

\newcommand\Hom{\operatorname{Hom}}

\newcommand\beq[1]{\begin{equation}\label{#1}}
\newcommand\eeq{\end{equation}}
\newcommand\beqa{\begin{eqnarray*}}
\newcommand\eeqa{\end{eqnarray*}}

\makeatletter \@addtoreset{equation}{section} \makeatother

\newtheorem{thm}[equation]{Theorem}
\newtheorem*{thm*}{Theorem}
\newtheorem{lem}[equation]{Lemma}

\newtheorem{prop}[equation]{Proposition}
\newenvironment{rmk}{\noindent\textbf{Remark}.}{\\}

\DeclareMathOperator{\RHom}{RHom}

\renewcommand{\H}{\mathcal{H}}
\usepackage{multicol}

\newcommand{\sC}{\mathscr{C}}

\DeclareMathOperator{\PGL}{PGL}
\DeclareMathOperator{\pt}{(pt)}

\newcommand{\undy}[1]{\,\underline{\!#1}}
\newcommand{\thiswasS}{\P(W)\times \P^1}
\DeclareMathOperator{\Sym}{Sym}

\title{Derived equivalent Calabi-Yau 3-folds from cubic 4-folds}
\author[J R Calabrese and R P Thomas]{John R Calabrese and Richard P Thomas}

\usepackage{verbatim}

\newcommand{\ee}{{\hat{e}}}
\newcommand{\ebe}{{\hat{e}}}
\begin{document}

\begin{abstract} \noindent
We describe pretty examples of derived equivalences and autoequivalences of Calabi-Yau threefolds arising from pencils of cubic fourfolds.
The cubic fourfolds are chosen to be \emph{special}, so they each have an associated K3 surface.
Thus a pencil gives rise to two different Calabi-Yau threefolds: the associated pencil of K3 surfaces, and the baselocus of the original pencil -- the intersection of two cubic fourfolds.
They both have crepant resolutions which are derived equivalent.
%
%We describe two pretty examples of pairs of derived equivalent Calabi-Yau 3-folds arising from pencils of special cubic 4-folds.
%
%The examples both begin with an intersection of two cubic 4-folds with different singularities. They both admit a crepant resolution and a derived equivalence to a K3-fibred Calabi-Yau 3-fold.
\end{abstract}
\maketitle

\tableofcontents

% \makeatletter
% \providecommand\@dotsep{5}
% \makeatother
% \listoftodos\relax
% \todo[inline]{remove todo list}

%%%%%%%%%%%%%%%%%%%%%%%%%%%%%%%%%%%%%%%%%%%%%%%%%%%%%%%%%%%%%%%%%%%%%%%%%%%

\section{Statement of result} \label{1}
We exhibit two pairs of derived equivalent Calabi-Yau 3-folds $(X,Y)$.
In both examples, $X$ is a crepant resolution of a complete intersection, while $Y$ is K3-fibred.
In the first example the equivalence is twisted by a Brauer class on $Y$, but in the second example there is no twisting.
The first pair both have Betti numbers $b_2 = \text{2}$, $b_3 = \text{126}$ (and so Euler charcteristic $-\text{120}$); the second pair have Betti numbers $b_2 = \text{2}$, $b_3 = \text{134}$ and Euler characteristic $-\text{128}$.
In this section we state the results; in Section \ref{back} we explain the motivation. \medskip

For the first example, we start by describing $X$.
Consider a \emph{generic} pencil of cubic fourfolds containing a fixed plane $P\subset\PP^5$.
Let $X_0$ be the baselocus of the pencil -- a (3,3) complete intersection in $\PP^5$, smooth except for 12 ordinary double points (ODPs) on $P$.
It admits a projective Calabi-Yau small resolution $X$ given by blowing up the plane $P$:
$$
X:=\Bl_PX_0.
$$
To describe $Y$ we choose another plane $\PP^2\subset\PP^5$ disjoint from $P$.
For any cubic $C$ in the pencil, projection from $P$ to $\PP^2$ makes $\Bl_PC$ into a quadric fibration over $\PP^2$.
As the cubic $C$ varies through the pencil we get a quadric fibration over $\PP^2\times\PP^1$, degenerate along a $(6,4)$ divisor $D$ with 66 OPDs.

Let $Y_0$ denote the double cover of $\PP^2\times\PP^1$ branched over $D$.
This has 66 ODPs over the ODPs of $D$, so let
$$ Y \to Y_0 \onto \P^2 \times \P^1 $$
% $$
% \xymatrix@=20pt{Y \ar[r]& Y_0\! \ar@{->>}[r]& \PP^2\times\PP^1}
% $$
denote any small resolution (we show they are all non-K\"ahler).
The quadric fibration endows $Y$ with a Brauer class $\alpha \in H^2_{\text{\'et}}\left(Y,\O^\times_Y\right)$.
\begin{thm*}
There is an equivalence $D(X) \simeq D(Y,\alpha)$.
\end{thm*}

The second example involves only projective Calabi-Yaus with no twisting.
It also comes from the data of a pencil of cubic fourfolds, this time required to all have an ODP at a fixed point $0\in\PP^5$.
Blowing up their baselocus $X_0$ in the singular point $0$ gives a smooth Calabi-Yau 3-fold
$$
X=\Bl_0X_0.
$$
The pencil also carries a universal hypersurface $\H \subset\PP^5\times\PP^1$.
Projecting from $\{0\}\times\PP^1\subset \PP^5\times\PP^1$ to a disjoint $\PP^4\times\PP^1$ gives a birational map
$$
\Bl_{\{0\}\times\PP^1}(\H)\To\PP^4\times\PP^1.
$$
This exhibits the left hand side as $\Bl_Y(\PP^4\times\PP^1)$, where $Y$ is the smooth intersection of a $(2,1)$ divisor and a $(3,1)$ divisor in $\PP^4\times\PP^1$.
It is therefore a Calabi-Yau 3-fold.

\begin{thm*}
There is a derived equivalence $D(X)\cong D(Y)$.
\end{thm*}

Nick Addington and Paul Aspinwall pointed out that in this example $X$ and $Y$ are birational, and so already derived equivalent. Our derived equivalence is very different, as we explain in Section \ref{odpeg}, so combined with the birational equivalence we get an exotic autoequivalence of $D(X)$.

\smallskip\noindent\textbf{Acknowledgements.}
This paper is an extended exercise in Kuznetsov's beautiful ideas about derived categories of cubic fourfolds \cite{Cu}, fractional Calabi-Yau categories \cite[Section 4]{V14}, and Homological Projective Duality (HPD) \cite{HPD}. Our debt to him is clear.

We would like to warmly thank Nick Addington and Ed Segal for useful discussions.
In addition, JC is grateful to Roland Abuaf, Michele Bolognesi, Enrica Floris, J{\footnotesize$\otimes$}rgen Rennemo and Brian Lehmann for helpful conversations on topics related to this paper.
Finally, we would like to thank the referee for helpful comments.

RT was partially supported by EPSRC programme grant EP/G06170X/1 and JC was partially supported by NSF RTG grant 1148609.

% \medskip
% Another example of derived equivalent Calabi-Yau 3-folds which are not birational to each other is given in \cite{CB}. This is also shown to fit within the framework of HPD in \cite{KuGr}.
% \medskip

\smallskip\noindent\textbf{Conventions.}
We work over the field of complex numbers $\C$.
By $D(X)$ we mean the \emph{bounded} derived category of \emph{coherent} sheaves on $X$.

\section{Background}\label{back}
This section is purely motivational, and the reader can safely skip to the two examples in Sections \ref{main} and \ref{odpeg}.
There we prove the results from first principles. \medskip

\subsection*{The short version}
A very brief summary is that we use the simplest possible case of HPD: the close relationship between the derived category of the baselocus of a pencil of hypersurfaces and the derived category of the universal hypersurface $\H\to\PP^1$ over the pencil.

We apply this to pencils of special cubic fourfolds, whose derived categories are very close to those of K3 surfaces \cite{Cu}.

The upshot is a relation between the derived category of the intersection of two cubic fourfolds, and the derived category of the K3 fibration over $\PP^1$ associated to the universal family of cubic fourfolds over the pencil.

Choosing the cubics to be special (so that we can really associate K3 surfaces to them) makes both sides of the above description singular.
We find examples where this issue can be resolved (crepantly).

\subsection*{The longer version}
A smooth hypersurface $H \subset \P^n$ of degree $d<n$ has derived category
\beq{fracCY}
D(H)=\big\langle\A_H,\O_H(d),\,\O_H(d+1),\,\ldots,\O_H(n)\big\rangle,
\eeq
where $\O_H(d),\ldots,\O_H(n)$ is an exceptional collection, and $\A_H$ -- the ``interesting part" of $D(H)$ -- is its right orthogonal:
\begin{eqnarray*}
\A_H \!&:=&\! \big\langle \O_H(d),\,\ldots,\O_H(n)\big\rangle^\perp \\ 
&=&\!\{E\in D(H)\colon \RHom(\O_H(i),E)=0\ \text{ for }\ i=d,\ldots,n\}.
\end{eqnarray*}

The category $\A_H$ is a ``fractional Calabi-Yau category" of dimension $(n+1)(1-2/d)$; that is, it has a Serre functor $S_{\A_H}$, some power of which is just a shift:
$$
S_{\A_H}^d\cong[(n+1)(d-2)].
$$

\medskip In the first case $d=1$, the category $\A_H$ is empty: the exceptional collection already generates $D(H)$, by Be{\u\i}linson's theorem.

The next case is $d=2$, i.e.~$H$ is a smooth quadric.
Kapranov was the first to show that $\A_H$ is a zero dimensional Calabi-Yau category: in fact it is equivalent to the derived category of 1 or 2 points, generated by its 1 or 2 \emph{spinor sheaves} when $n$ is even or odd respectively.

The next case, $d=3$, is poorly understood in general (though, for any $d$ there is a highly non-commutative description in terms of $A_\infty$-algebras \cite{BDFIK}).
One interesting example is when $n=5$ so that $H$ is a cubic fourfold.
In this case $\A_H$ is $CY_2$, namely the Serre functor is the shift by two.
For special cubic fourfolds, $\A_H$ is equivalent to $D(K3)$, for a genuine K3 surface (possibly with a twist by a Brauer class) \cite{Cu}.
For the generic cubic, $\A_H$ cannot be the derived category of a variety (for example the rank of its numerical Grothendieck group is too small).
As it is a deformation of some $D(K3)$ and has the same Hochschild (co)homology, one usually calls $\A_H$ a \emph{non-commutative} $K3$ surface.

Our first example comes from a combination of the cases $(d,n) = (2,3)$ and $(d,n) = (3,5)$.

\subsection{Pencils of cubic fourfolds}
Now suppose we have a \emph{pencil} $\PP^1$ of cubic fourfolds.
Denote by
\begin{equation}\label{univ}
        \begin{tikzcd}
                \H \ar{r}{\pi} & \P^1 
        \end{tikzcd}
\end{equation}
% \beq{univ}
% \mathcal H\stack\To\PP^1
% \eeq
the universal cubic 4-fold over $\PP^1$: for any $t \in \P^1$, the fibre $\H_t$ is the corresponding cubic fourfold of the pencil.
The total space $\H$ is a $(3,1)$ divisor in $\P^5 \times \P^1$.
A family version of the previous discussion (over the base $\P^1$) gives us a semi-orthogonal decomposition
\begin{equation}\label{semi} 
        D(\H) = \big\langle \A_\H\, , \, \pi^*D(\P^1)(i,0) : i= 3,4,5 \big\rangle
\end{equation}
where the ``interesting part'' $\A_\H$ can be rewritten as
$$ \A_\H = \big\langle \O_{\mathcal H}(i,j) : i=3,4,5,\,j=0,1\big\rangle^\perp.$$
One verifies that $\A_\H$ is a $CY_3$-category, i.e.~a non-commutative Calabi-Yau threefold.
Using $\pi$, we can view $\A_\H$ as a $CY_2$-fibration with base $\P^1$, whose fibre at $t \in \P^1$ is the non-commutative $K3$ surface $\A_{\H_t}$.

\medskip
\noindent{\bfseries Digression.}
        We take a moment to make this notion less vague.
        The category $D(\P^1)$ has a tensor product.
        Via pullback along $\pi$, any object $E \in D(\H)$ can be tensored with an object of $D(\P^1)$.
        We might think of $D(\H)$ as being a module over the commutative algebra $D(\P^1)$.
        In fact, this becomes literally true in the ($\infty$-)category of stable ($\infty$-)categories.
        % For simplicity, let's assume everything to be smooth.
        
        Given a point $t \in \P^1$, we have the fibre $\H_t \subset \H$.
        As the (derived) pullback $t^*\colon D(\P^1) \to D\pt$ preserves the tensor product, we can think of it as giving a homomorphism of algebras in this sophisticated category of stable categories.
        We can also pull back $D(\P^1)$-modules along $t$ and we have $D\pt \otimes_{D(\P^1)} D(\H) = \text{Perf}(\H_t) \subset D(\H_t)$, where the latter is the subcategory of \emph{perfect complexes}, which coincides with the whole $D(\H_t)$ when $\H_t$ is smooth.
        Now that all this is in place, let's go back to $\A_\H$.
        
        The decomposition \eqref{semi} is \emph{$D(\P^1)$-linear}, in the sense that tensoring with $D(\P^1)$ preserves each component.
        In other words, any component $\mathcal C$ has the structure of a $D(\P^1)$-module.
        
        In particular, the module $\A_\H$ can be pulled back along any $t \in \P^1$.
        It is shown in \cite[Prop. 5.1]{kusod} that $D\pt \otimes_{D(\P^1)} \A_\H$ is in fact $\A_{\H_t} \cap \text{Perf}(\H_t)$.

\bigskip
Homological projective duality relates the derived categories of universal hypersurfaces over linear systems with the derived categories of their base loci.
The base locus of the pencil is a (3,3) Calabi-Yau threefold complete intersection $X_{3,3}$ in $\PP^5$, and by HPD it's derived equivalent to $\A_{\mathcal H}$.
\begin{prop} \label{prop}
$D(X_{3,3})\simeq \A_{\mathcal H}.$
\end{prop}
Since pencils form the most elementary case of HPD, the equivalence can
be proved directly, at least when $\H$ and $X_{3,3}$ are smooth, using Orlov's decomposition \cite{Or}.
\begin{proof}
The total space of the universal hypersurface \eqref{univ} is just the blow up of $\PP^5$ in the baselocus:
\beq{shesgonnablow}
\Bl_{X_{3,3}}\PP^5\ \cong\ \H.
\eeq
Using the notation
\begin{center}
\begin{tikzcd}
        \hspace{-1.95cm}\P^1 \times X_{3,3} \cong E \ar{d}{p} \ar[hook]{r}{\iota}  & \H \\
        X_{3,3} &
\end{tikzcd}    
\end{center}
for the exceptional divisor of \eqref{shesgonnablow}, the equivalence is given by the composition
\beq{compost}
D(X_{3,3})\Rt{p^*}D(E)\Rt{\iota_*}D(\H)\To\A_{\H\,},
\eeq
where the last arrow is projection (the left adjoint of $\A_{\H}\into D(\H)$).

In fact, one may start with Orlov's decomposition
$$
      D(\H) = \Langle \O, \O(1,0), \O(2,0), \O(3,0), \O(4,0), \O(5,0), \,\iota_*p^*D(X_{3,3}) \Rangle
$$
and mutate away.\footnote{We are grateful to A.~Kuznetsov for indicating which mutations to make here.}
By right mutating the first three terms to the end, thus tensoring them with $K_{\H}^{-1} = \O(3,1)$, one obtains
$$
\Langle\O_{\mathcal H}(3,0),\O_{\mathcal H}(4,0),\O_{\mathcal H}(5,0),\,
\iota_*p^*D(X_{3,3}),\,\O_{\mathcal H}(3,1),\O_{\mathcal H}(4,1),\O_{\mathcal H}(5,1)\Rangle.
$$
To project to the right orthogonal of all of these sheaves, we left mutate $\iota_*p^*D(X_{3,3})$ past $\O_{\mathcal H}(3,0),\O_{\mathcal H}(4,0),\O_{\mathcal H}(5,0)$. Therefore the composition \eqref{compost} is $\mathbb L_{\O_{\mathcal H}(3,0)}\mathbb L_{\O_{\mathcal H}(4,0)}\mathbb L_{\O_{\mathcal H}(5,0)}
\iota_*p^*D(X_{3,3})$, giving an isomorphism to
$$
\A_{\mathcal H}=\Langle\O_{\mathcal H}(3,0),\O_{\mathcal H}(4,0),\O_{\mathcal H}(5,0),
\O_{\mathcal H}(3,1),\O_{\mathcal H}(4,1),\O_{\mathcal H}(5,1)\Rangle^\perp
$$
which is what we needed to conclude.
\end{proof}

\subsection{Cubic fourfolds containing a plane} Proposition \ref{prop} says we can see $X_{3,3}$ as a noncommutative K3 fibration.
To bring things down to earth, we would like to choose an example where the K3 fibration is commutative.

Examples of cubic fourfolds $H$ for which $\A_H$ is really the derived category of a geometric K3 surface are given in \cite{Cu}.
The easiest is to take $H$ to contain a plane.

Write $\PP^5=\PP(V\oplus W)$ where $V,W$ are copies of $\C^3$, and consider a cubic $H$ containing the plane $\PP(V)$.
We can blow up the plane $\P(V)$ and project away from it onto $\P(W)$, producing a $\P^3$-bundle $\Bl_{\PP(V)}\PP^5\to\PP(W)$.
Taking the intersection of a fibre with the preimage of $H$ gives a cubic surface containing $\PP(V)$; it is therefore the union of $\PP(V)$ and a quadric.
Intersecting with the proper transform of $H$ instead removes $\PP(V)$, so that
\beq{qfib}
\Bl_{\PP(V)}H \stackrel{}\To\PP(W)
\eeq
is a quadric surface fibration over two-dimensional projective space $\P(W)$.

We think of this as being a family version, with base $\P(W)$, of the $d=2$, $n=3$ case of equation \eqref{fracCY}.
The (interesting part of the) derived category of a smooth fibre of \eqref{qfib} is the derived category of two points.
As we move along the base, these two points vary, describing a double cover $S \onto \P(W)$ branched along the locus of singular fibres.
This locus is a sextic curve and $S$ is a K3 surface.

Thinking of $S$ as a moduli space of spinor sheaves on the fibres of \eqref{qfib}, its product with $\Bl_{\P(V)}H$ carries, analytically locally over $S$, a universal sheaf.
On overlaps the sheaves glue, up to invertible scalars.
Since the gluings might not satisfy the cocycle condition, they define a Brauer class $\alpha$ on $S$.
The universal sheaf then exists as a $p^*\alpha^{-1}$-twisted sheaf (where $p\colon \Bl_{\P(V)}H \times{S} \to S$ is the projection); using it as a Fourier-Mukai kernel produces an embedding $D(S,\alpha) \to D(\Bl_{\P(V)}H)$, where the former is the derived category of $\alpha$-twisted sheaves.
On the other hand, $\A_H$ embeds in $D(\Bl_{\P(V)}H)$ by the blow up formula.
By performing a series of mutations, Kuznetsov shows that indeed $D(S,\alpha) \simeq \A_H$ \cite[Theorem 4.3]{Cu}.
Therefore in this case we see the K3 category arising from the commutative K3 surface $S$. \medskip

If we now start with a whole pencil $\H$ of cubics containing $\P(V)$, we would like to have a family version of the previous discussion, over the $\P^1$ base of the pencil.
Said differently, we would like to realise commutatively a special case of Proposition \ref{prop}.
There is a hitch, however.
The baselocus $X_{3,3}$ of a generic pencil of cubics containing the plane $\P(V)$ has twelve ODPs.
Dually, the universal hypersurface $\mathcal H$ \eqref{univ} is a fivefold with twelve ODPs as well.
For instance, this implies that Orlov's blow up formula breaks down.
To remedy this, we need to somehow resolve the singularities.
This is why in Section \ref{main} we start over, working with the right blow up of $\PP^5$ from the beginning.

\subsection{Cubic fourfolds with an ODP} Let $H$ be a cubic fourfold with a single ODP at $0\in\PP^5$.
Projecting to a disjoint $\PP^4\subset\PP^5$ gives a map
$$
\Bl_0H\To\PP^4
$$
which is degree 1 and so birational.
It exhibits the left hand side as 
$$
\Bl_0H\ \cong\ \Bl_S(\PP^4),
$$
where $S$ is a $(2,3)$ complete intersection $K3$ surface in $\PP^4$.
This correspondence gives a derived equivalence\footnote{
        The singularity in $H$ means we have to be more careful in defining $\A_H$; see \cite{Cu} for details.
        We will avoid this issue in Section \ref{odpeg}.
        }
\cite{Cu}
$$
\A_H\ \cong\ D(S).
$$
So we can play the same trick again, relating the baselocus of a pencil of such cubics with the associated $\PP^1$ family of K3 surfaces $S$.
Again the results are singular on both sides but can be resolved.
In Section \ref{odpeg} we work in $\Bl_0\PP^5$ from the beginning to avoid these singularities.

\section{First Example} \label{main}
Fix two copies $V,W$ of $\C^3$ and write $\P^5 = \P(V \oplus W)$.
Projecting from the plane $\P(V) \subset \P^5$ to the plane $\P(W)$ gives the following diagram.
\begin{center}
\begin{tikzcd}
        Z:= \Bl_{\P(V)} \P^5 \ar{r}{\rho} \ar{d}[swap]{\pi} & \P(W) \\
        \P^5 &
\end{tikzcd}    
\end{center}
The map $\rho$ is a $\P^3$-bundle.
More precisely, $\rho$ is the projective completion  of $\undy{V}(1)$ over $\PP(W)$:
\beq{projb}
\PP\big(\undy{V}(1)\oplus\O_{\PP(W)}\big)\Rt{\rho}\PP(W).
\eeq
Here $\undy{V}$ denotes the trivial bundle $V\otimes_\C\O_{\PP(W)}$.
As the projectivisation of a vector bundle, $Z$ carries a tautological line bundle which we denote by $\O_\rho(-1)$.
A computation gives
\beq{linebun}
\pi^*\O_{\PP^5}(3)(-E)\ \cong\ \O_\rho(2)\otimes\rho^*\O_{\PP(W)}(3)
\eeq
where $E\cong\PP(V)\times\PP(W)$
 is the exceptional divisor of $\pi$. 
In particular,
\begin{align}
        \label{NW}
        \rho_*(\pi^*\O_{\PP^5}(3)(-E)) &\cong \rho_* \O_\rho (2) \otimes \O_{\P(W)}(3) \\
                                                \notag&\cong S^2 \left( \undy{V}^*(-1) \oplus \O_{\P(W)} \right) \otimes \O_{\P(W)}(3) \\
                                                \notag&\cong \left( 
                                                        \left( S^2\undy{V}^* \right) (-2) \oplus
                                                        \undy{V}^*(-1) \oplus
                                                        \O_{\P(W)} \right) \otimes \O_{\P(W)}(3) \\
                                                        &\cong \left( S^2 \undy{V}^* \right) (1)\oplus \undy{V}^*(2)\oplus\O_{\P(W)}(3). \notag   
\end{align}
Taking global sections yields
\begin{equation}\label{lineb}
\begin{tikzcd}
        H^0( Z, \pi^* \O_{\P^5}(3)(-E) ) =
        H^0 \left( \P^5, \I_{\P(V)}(3) \right)
        % \ar[color=white]{r}[color=black,description]{\text{\Large$\,\subset$}}
        \ar[hook]{r}
        \ar[-,double,double distance=1pt]{d}
        & H^0\big( \P^5, \O_{\P^5}(3) \big) 
        \ar[-,double,double distance=1pt]{d}
        \\
        \left( S^2 V^* \otimes W^* \right) \oplus 
        \left( V^* \otimes S^2 W^* \right) \oplus 
        \left( S^3 W^* \right)
        % \ar[color=white]{r}[color=black,description]{\text{\Large$\subset$}}
        \ar[hook]{r}
        & S^3\left( V^* \oplus W^* \right)
\end{tikzcd}
\end{equation}
using the fact that $\pi_* \O_{\P^5}(-E) = \I_{\P(V)}$.
% \beq{lineb}
% \xymatrix@=10pt{
% H^0\big( Z,\pi^*\O_{\PP^5}(3)(-E)\big)\ =\ H^0\big({\PP^5},\I_{\PP(V)}(3)\big) \ar@{=}[d]\ar@{}[r]|-*[@]{\subset} & H^0\big({\PP^5},\O_{\PP^5}(3)\big) \ar@{=}[d] \\
% (S^2V^*\otimes W^*)\oplus(V^*\otimes S^2W^*)\oplus(S^3W^*)
% \ar@{}[r]|-*[@]{\subset} & S^3(V^*\oplus W^*).}
% \eeq

\smallskip
\begin{lem}
        The line bundle $\pi^*\O_{\PP^5}(3)(-E)$ has no baselocus.
\end{lem}
\begin{proof}
This can be deduced from \eqref{lineb} as follows.
It is clear that the baselocus of the linear system $H^0(\P^5,\I_{\P(V)}(3)) \subset H^0(\P^5,\O_{\P^5}(3))$ is the plane $\P(V)$.
Therefore the baselocus of $\pi^*\O_{\PP^5}(3)(-E)$ must be contained in the exceptional divisor $E\cong\PP(V)\times\PP(W)$.
But
$\pi^*\O_{\PP^5}(3)(-E)|_E\cong\O_{\PP(V)}(2)\boxtimes\O_{\PP(W)}(1)$, which has sections $S^2V^*\otimes W^*$ (without baselocus).
These are in turn surjected onto by the left hand side of \eqref{lineb}.
\end{proof}
Therefore, by Bertini's theorem we can pick a pencil
$$
\P^1 = \langle s_0, s_\infty\rangle \,  \subset \, H^0\left( Z, \pi^*\O(3)(-E) \right)
% \P(L) \PP^1=\langle s_0,s_\infty\rangle\ \subset\ |\pi^*\O(3)(-E)| = H^0\left(\Bl_{\P(V)}\P^5,\pi^*\O(3)(-E)\right)^*
$$
whose baselocus
\beq{Xdef}
X:=\{s_0=0=s_\infty\}\subset Z %\Bl_{\PP(V)}\PP^5
\eeq
is smooth, and over which the universal hypersurface
\begin{equation}\label{unihyp}
        \begin{tikzcd}
                \Bl_X Z \,\,\, \cong \hspace{-1.2cm} & \H \ar{d} \,\,  &\hspace{-1.3cm} = \{xs_0 + ys_\infty = 0 \} \ar[hook]{r} & Z \times \P^1 \\
                & \P^1 & &
        \end{tikzcd}
\end{equation}
is also smooth.
Since the anticanonical bundle of $Z$ is $\pi^*\O(6)(-2E)$ we see that $X$ is a Calabi-Yau 3-fold.

\begin{rmk}
	The projection of $X$ to $\PP^5$ is also easily seen to be the $(3,3)$ complete intersection $X_0$ of Section \ref{1}.
The point is that $X\cap E$ is the intersection of two $(2,1)$ divisors in $\PP(V)\times\PP(W)$, which is generically a section over $\PP(V)$ but which has fibre $\PP^1$ over 12 points of $\PP(V)$.
In other words it is $\Bl_{\text{12}}\PP(V)$ and the projection to $\PP^5$ blows this back down to $\PP(V)$.	
\end{rmk}

Using the projective bundle structure (\ref{projb}, \ref{linebun}) we see that $\mathcal H$ is an element of the linear system
$$
\left|\O_\rho(2)\otimes\rho^*\O_{\PP(W)}(3)\boxtimes\O_{\P^1}(1)\right|
$$
on the $\PP^3$-bundle $Z\times\P^1\Rt{\rho\times\!1}\PP(W)\times\P^1$.
Since this has degree two on the fibres, \emph{$\H$ is a quadric fibration over $\thiswasS$}.

Via the isomorphism \eqref{NW} we think of $\mathcal H$ as being defined by a section of $S^2(\undy{V}^*\oplus\O_{\PP(W)}(1))(1,1)$.
That is, a quadratic form on the fibres of $\undy{V}\oplus\O_{\PP(W)}(-1)$, twisted by $\O_{\thiswasS}(1,1)$.
It is generically of rank 4 on the fibres, but drops to rank $\leq 3$ on the divisor $D$ where its determinant in
$$
\left( \Lambda^4\left(\undy{V}^*\oplus\O_{\PP(W)}(1)\right)\right)^{\otimes2}(4,4) \, \cong \, \O_{\thiswasS}(6,4)
$$
vanishes.
It further drops to rank $2$ at
$$
4\big(c_1(E\otimes N)c_2(E\otimes N)-c_3(E\otimes N)\big)
$$
points \cite{HaTu}, where $E$ is the bundle $\undy{V}^*\oplus\O_{\PP(W)}(1)$ and $N$ is the $\Q$-line bundle $\O_{\thiswasS}(1/2,1/2)$.
With some work, this can be computed to be 66.
At these 66 points, the divisor $D$ necessarily has an ODP.

\medskip
We now have all the ingredients in place to cook up our desired equivalence.
As $\H \to \thiswasS$ is a quadric fibration, there is a corresponding \emph{even Clifford algebra} sheaf $\sC_0$ on $\thiswasS$ -- see \cite[Section 1.5]{ABB} for its definition and see also \cite[Proposition A.1]{ABB}, where it is proved that $\sC_0$ is isomorphic to $\mathfrak{B}_0$ of \cite[Section 3]{kuquad}.
The relative version of Kuznetsov's HPD for quadrics found in \cite{ABB} provides the first step of our desired equivalence.
\begin{lem}
        There is an equivalence
        \begin{equation} D(X) \cong D\left(\thiswasS, \sC_0\right) \end{equation}
        where by the latter we mean the bounded derived category of right coherent $\sC_0$-modules.
\end{lem}
\begin{proof}
        This follows from \cite[Theorem 2.19 (2)]{ABB}, which fits in the general framework of relative HPD \cite[Theorem 2.16]{ABB}.
        To use their notation: the base scheme $Y$ is our $\P(W)$, the vector bundle $E$ is $\undy{V}(1)\oplus \O_{\P(W)}$, the base of the family of quadrics $S$ is our $\thiswasS$, the base locus $X$ is our $X$, the total family of quadrics $Q$ is our $\H$, $m = 2$ and $n = 4$.
        
        Our pencil $\langle s_0, s_\infty\rangle$ corresponds to a two-dimensional subspace of the global sections of $\pi^*\O(3)(-E)$, which we identify with the bottom-left corner of \eqref{lineb}.

To satisfy the assumptions of \cite[Theorem 2.16]{ABB}, we need to check that the map
\begin{equation}\label{ABBcheck}
\begin{tikzcd}
	\O_{\P(W)}^{\oplus 2} \ar{r}{(s_0,s_\infty)} & \rho_* \pi^* \O(3)(-E)
\end{tikzcd}
\end{equation}
defined by the pencil has rank 2 at every point.
(In the notation of \cite{ABB}, $L = \O_{\PP(W)}(-3)^{\oplus 2}$, $V = S^2 \left(\undy{V}(1) \oplus \O_{\PP(W)}\right)$.) It is sufficient to show that the composition of \eqref{ABBcheck} with the projection $\rho_* \pi^*\O(3)(-E)\to
\big(S^2 \undy{V}^* \big)(1)$ to the first summand of \eqref{NW} has rank 2 at every point.
        
We see this for generic $s_0, s_\infty$ as follows.
Consider their projections to the first term $S^2V^* \otimes W^*$ of the bottom-left corner of \eqref{lineb}.
They define two elements of $\Hom(W, S^2V^*)$ whose images intersect only in $0$ (by genericity and the observation that $\dim S^2V^*=2\dim W$).
That is, for each point of $\PP(W)$, we get two quadratic forms on $V$ (twisted by $\O(1)$) which are \emph{distinct}.
This is the required condition.
\end{proof}

To make the right hand side of the equivalence (more) geometric, we use the work of \cite{Kuf} (we thank Alexander Kuznetsov for highlighting this paper).
If $s \in \thiswasS$, we know what the fibre of $\H$ over it looks like:
\begin{enumerate}
        \item generically it is a smooth quadric $\H_s \cong \P^1\times \P^1$;
        \item over the smooth locus of $D$, $\H_s$ is a quadric cone;
        \item over the ODPs of $D$, $\H_s$ is the union of two planes intersecting in a line.
\end{enumerate}
All the fibres contain lines and we can consider the moduli space parameterising them.
Let $F \to \thiswasS$ be relative Fano scheme of lines of $\H \to \thiswasS$ (the derived category of which was studied in \cite{Kuf}).
Once again, the fibres of $F$ over $\thiswasS$ are explicit:
\begin{enumerate}
        \item if $s \notin D$, $F_s$ is the disjoint union of two smooth lines;
        \item for $s$ in the smooth locus of $D$, $F_s$ is topologically a single smooth conic;
        \item for $s$ a singular point of $D$, $F_s$ is the union of two planes $\Sigma^{+}_s, \Sigma^{-}_s$ intersecting at a point.
\end{enumerate}
Away from the singularities of $D$ we see that $F$ is a $\P^1$-bundle over a double cover of $\thiswasS$.
More precisely, let $F \to Y_0 \to \thiswasS$ be the Stein factorisation of $F \to \thiswasS$.
The morphism $Y_0 \to \thiswasS$ is a double cover branched at $D$, while $F \to Y_0$ is generically a $\P^1$-bundle.

Let us now choose a plane $\Sigma_s \in \{\Sigma^+_s, \Sigma^-_s\}$ in each fibre over the singular locus of $D$.
In \cite[Proposition 4.4]{Kuf} it is shown that the flip $F'$ in all the planes $\Sigma_s$ factors as a composition
$$F' \to Y \to Y_0$$
where $Y \to Y_0$ is a small resolution and $F' \to Y$ is a $\P^1$-bundle.
In Lemma \ref{nick} we show that one \emph{cannot} choose the $\Sigma$s so that $Y$ is a projective variety.

Let $\alpha \in H^2_{\text{\'et}}(Y,\O^\times_Y)$ be the Brauer class coming from $[F'] \in H_{\text{\'et}}^1(Y,\PGL_2)$.
It is shown in \cite[Lemma 5.7]{Kuf} that there exists an Azumaya algebra $\mathscr{A}$ on $Y$, representing the class $\alpha$, such that $\sigma_* \mathscr{A} \cong \sC_0$, where $\sigma$ is the composition $Y \to Y_0 \to \thiswasS$ (recall once again \cite[Proposition A.1]{ABB}).
Kuznetsov goes on to show \cite[Proof of Theorem 1.1]{Kuf} that $\sigma^*$ induces an equivalence between the derived categories of $\mathscr{A}$-modules and $\sC_0$-modules. 
As a consequence, we now have our desired equivalence:
\begin{equation}
        D(X) \cong D\left(\thiswasS, \sC_0\right) \cong D(Y,\mathscr{A}) \cong D(Y,\alpha)
\end{equation}
where the last term is the bounded derived category of coherent $\alpha$-twisted sheaves on $Y$.

\begin{rmk}
        An alternative approach to the equivalence could be realised by mimicking \cite{AddThesis, Add}.
        We can construct the small resolution $Y$ as a moduli space of spinor sheaves on the fibres of $\H$ as follows.
        
        Let $U$ be an analytic open subset of $\thiswasS$, small enough so we can pick a section $s_U$ of the quadric bundle passing through only smooth points of the fibres.
        We define $Y|_U$ to be the moduli space of lines in the quadric fibres intersecting the section $s_U$.
        Generically, the fibre is a smooth quadric $\P^1\times \P^1$ and there are two lines through the basepoint.
        So $Y|_U \to U$ is a double cover.
        Over those fibres where the quadric drops rank by 1, there is only one line, so the double cover branches.
        Over those fibres where the quadric drops rank by 2, there is a $\P^1$ of lines.
        This $\P^1$ is what gives the small resolution $Y \to Y_0$.
        
        Over an open set containing a corank-2 quadric we get a fibre which is the union of two planes.
                Changing which plane the section goes through flops $Y$ from one small resolution of $Y_0$ to the other. 
                
        The spinor sheaf is then defined to be the ideal sheaf of the line in the quadric fibre.
        Over overlaps the section changes and so the line with it, but the ideal sheaf remains isomorphic away from the corank-2 quadric fibres.
        Since there are only finitely many of these fibres, we can choose our cover so that they do not lie on overlaps.
        With these choices, the $Y|_U$ glue up uniquely, while the spinor sheaves glue modulo scalars giving a universal sheaf twisted by a Brauer class (cf.~\cite{AddThesis}).
\end{rmk}

Finally, just as in \cite{AddThesis} we have the following.
\begin{lem}\label{nick} The threefold $Y$ is non-K\"ahler (and a fortiori non-projective).
\end{lem}

\begin{proof} We thank Nick Addington for the following argument.

The Brauer class $\alpha$ is represented by the $\P^1$-bundle $F' \to Y$.
By \cite{Be} we have the semiorthogonal decomposition
\begin{align*}
        D(F')   &\cong \left\langle D(Y), D(Y,\alpha) \right\rangle \\
                        &\cong \left\langle D(Y), D(X) \right\rangle.
\end{align*}
Hence,
\beq{dim}
        \dim H^{\mathrm{even}}(F',\Q) = \dim H^{\mathrm{even}}(Y, \Q) + \dim
        H^{\mathrm{even}}(X, \Q).
\eeq
On the other hand, $c_1(K_{F'/Y})$ generates the cohomology of the fibre $\PP^{1}$ over $\Q$, so we can apply the Leray-Hirsch theorem to deduce that
$$
H^*(F', \Q)\ \cong\ H^*(Y, \Q) \otimes H^*(\P^1, \Q)
$$
as vector spaces.
In particular, $\dim H^{\mathrm{even}}(F',\Q)=2\dim H^{\mathrm{even}}(Y, \Q)$. Combined with \eqref{dim} gives $\dim H^{\mathrm{even}}(Y, \Q)\ =\ \dim H^{\mathrm{even}}(X, \Q)$ or, equivalently,
$$
b_2(Y)\,=\,b_2(X).
$$
Now a smooth (3,3) complete intersection in $\PP^5$ has $b_2=1$ by the Lefschetz hyperplane theorem.
Degenerating to $X_0$ with ODPs only changes $b_3$, and passing to the small resolution $X$ by blowing up the plane adds the defect to $b_2$.
Here the defect is the number of relations in $H_3$ between the vanishing cycles of $X_0$.
In this case, there is only one given by the plane $\PP^2$.
Hence
$$
b_2(X)=2.
$$
Now both $Y_0$ and $Y$ have two independent $H^2$ classes pulled back from $\thiswasS$.
Since $b_2(Y)=2$ we see that it has no extra classes; its $H^2$ is pulled back from $Y_0$.
In particular, the exceptional curves of the small resolution are all trivial in (co)homology, so $Y$ cannot be K\"ahler.
\end{proof}

\section{Second example} \label{odpeg}

For the second example, we work in the blow up of $\PP^5$ in a single point $0$, with exceptional divisor $e$:
\begin{equation}\label{Pdef}
        \begin{tikzcd}
                P:= \Bl_0\P^5 \ar{d} \ar[hookleftarrow]{r} & e  \\
                \P^5 & 
        \end{tikzcd}
\end{equation}
We consider divisors in the linear system of\,\footnote{We are suppressing the pullback maps from the notation.}
$$
K_P^{-1/2}\,=\,\O(3)(-2e).
$$
These are the proper transforms of cubic 4-folds in $\PP^5$ with an ODP at $0$. Picking a generic pencil
$$
\PP^1\subset\big|\O(3)(-2e)\big|
$$
of such divisors, their baselocus is a smooth Calabi-Yau 3-fold
\beq{X2def}
X\subset P=\Bl_0\PP^5.
\eeq

We use the universal hypersurface $\H\subset P\times\PP^1$ to define the dual Calabi-Yau threefold $Y$.

\begin{lem}
There is an isomorphism
\beq{birat}
\H\ \cong\ \Bl_Y(\PP^4\times\PP^1),
\eeq
where $Y\subset\PP^4\times\PP^1$ is the Calabi-Yau 3-fold intersection of a $(2,1)$ and a $(3,1)$ divisor.
\end{lem}

\begin{proof}
Consider the variety of lines in $\PP^5$ through the point $p$. This is a copy of $\PP^4$, over which $P$ is therefore a $\PP^1$-bundle:
\beq{p1bdl}
P=\PP\big(\O_{\PP^4}(1)\oplus\O_{\PP^4}\big)\Rt{\rho}\PP^4.
\eeq
Crossing with $\P^1$, we have the composition
\begin{center}
\begin{tikzcd}
        \H \ar[hook]{r} \ar[bend left = 15, start anchor = 50, end anchor = 150]{rr}{\rho|_\H} & P \times \P^1 \ar{r}{\rho} & \P^4 \times \P^1
\end{tikzcd}
\end{center}
which we will show is the claimed blow up (here we mildly abuse notation and write $\rho$ for $\rho \times 1$). \medskip 

The divisor $\H\subset P\times\PP^1$ is the zero locus of a section $s\_\H$ of the line bundle
\beq{lineconf}
\O(3)(-2e)\boxtimes\O_{\PP^1}(1).
\eeq
We need to express this in terms of the geometry of the bundle \eqref{p1bdl}.
Considering $P$ as the projectivisation of the vector bundle $\O(1)\oplus\O\to\PP^4$, it carries a tautological line bundle $\O_\rho(-1)$.
Its dual has a canonical section $(0,1) \in H^0(\PP^4,\O(-1)\oplus\O)$ whose zero locus is the divisor $e$, so
\beq{rholine}
\O_\rho(1)\ \cong\ \O(e).
\eeq

Considering $P$ as the blow up of $\PP^5$ in the point $0 \in \P^5$, which is the baselocus of the lines parameterised by $\PP^4$, we also have\footnote{To avoid ambiguity in our notation, we do \emph{not} suppress the pullback map $\rho^*$ even as we continue to omit the others.}
\beq{data} 
\rho^*\O(1)=\O(1)(-e).
\eeq
Therefore by \eqref{lineconf}, \eqref{rholine} and \eqref{data}, $\H$ lies in the linear system of
\beq{lines}
K_{P\times\PP^1}^{-1/2}\ \cong\ \O(3,1)(-2e)\ \cong\ \rho^*\O(3,1)(e)\ \cong\ \rho^*\O(3,1)\otimes\O_\rho(1).
\eeq
This has degree 1 on the $\PP^1$-fibres of $\rho$, so each fibre intersects $\H$ in either a point or the whole line.

Therefore $\rho|_\H$ is a birational map to $\P^4 \times \P^1$, and it only contracts $\P^1$s over the zero locus of $\rho_*s_\H$, which is a section of
\begin{align}\label{isoref}
        \rho_* \left( \rho^* \O(3,1))\otimes \O_\rho(1) \right)
        &= \O(3,1) \otimes \rho_* \O_\rho(1) \\ \notag
        &= \O(3,1) \otimes \big( \O(1,0) \oplus \O \big)^*
        = \O(2,1) \oplus \O(3,1).
\end{align}
This locus is thus a complete intersection of $(2,1)$ and $(3,1)$ divisors, as claimed.
\end{proof}

\begin{thm}\label{thm2}
There is a derived equivalence between $X$ \eqref{X2def} and $Y$ \eqref{birat},
$$
D(X)\ \cong\ D(Y).
$$
\end{thm}

\begin{proof}
We again use the fundamental relation
\beq{HPD1}
\H\ \cong\ \Bl_X(P).
\eeq
By Orlov's theorem \cite{Or} this gives a semi-orthogonal decomposition
\beq{semi1}
D(\H)\ \cong\ \Langle D(P),D(X)\Rangle,
\eeq
where the first term is embedded by pullback and the second by the functor of pulling back to the exceptional divisor and then pushing forwards into the blow up \eqref{HPD1}.

The description $\H\cong\Bl_Y(\PP^4\times\PP^1)$ of \eqref{birat} gives a similar semi-orthogonal decomposition
\beq{semi2}
D(\H)\ \cong\ \Langle\rho^*D(\PP^4\times\PP^1),D(Y)\Rangle.
\eeq
We will mutate \eqref{semi1} into \eqref{semi2} to get the equivalence $D(X)\cong D(Y)$. Our method is motivated by \cite[Section 5]{Cu}, heavily modified. We leave all the elementary sheaf cohomology calculations to the reader.

We start with the following semi-orthogonal decomposition of $D(P)$,
$$
\Langle\O,\O(1),\O(2),\O(3),\O(4),\O(5),\O_e,\O_e(-e),\O_e(-2e),\O_e(-3e)\Rangle
$$
obtained from Orlov's theorem \cite{Or} applied to yet another blow up -- the original one \eqref{Pdef}. Right mutating $\O(3),\O(4),\O(5)$ past $\O_e,\O_e(-e)$ turns this into 
$$
\Langle\O,\O(1),\O(2),\O_e,\O_e(-e),\O(3)(-2e),\O(4)(-2e),\O(5)(-2e),\O_e(-2e),\O_e(-3e)\Rangle.
$$
Then we left mutate the last 5 terms to the front of the exceptional collection, thus tensoring them with $K_P=\O(-6)(4e)$ to yield
\begin{multline*}
D(P)=\Langle\O(-3)(2e),\O(-2)(2e),\O(-1)(2e),\O_e(2e),\O_e(e), \\
\O,\O(1),\O(2),\O_e,\O_e(-e)\Rangle.
\end{multline*}

Substituting this into \eqref{semi1} and right mutating the first 5 terms all the way to the end tensors them with $K^{-1}_\H=\O(3,1)(-2\ee)$, where $\ee$ is the \emph{total} transform of $e \subset P$ in $\Bl_X P$.
Thus we can write $D(\H)$ as
\begin{multline*}
\Langle\O,\O(1,0),\O(2,0),\O_\ee,\O_\ee(-\ebe), \\
D(X),\O(0,1),\O(1,1),\O(2,1),\O_\ee(0,1),\O_\ee(-\ebe)(0,1)\Rangle.
\end{multline*}

Now right mutate $D(X)$ to the end, $\O(2,0)$ past $\O_\ee$, and $\O(2,1)$ past $\O_\ee(0,1)$, to give
\begin{multline*}
\Langle\O,\O(1,0),\O_\ee,\O(2,0)(-\ebe),\O_\ee(-\ebe), \\
\O(0,1),\O(1,1),\O_\ee(0,1),\O(2,1)(-\ebe),\O_\ee(-\ebe)(0,1),D(X)\Rangle.
\end{multline*}
Next left mutate the $3^\mathrm{rd}$ term past the $2^\mathrm{nd}$, the $5^\mathrm{th}$ past the $4^\mathrm{th}$, the $8^\mathrm{th}$ past the $7^\mathrm{th}$, and the $10^\mathrm{th}$ past the $9^\mathrm{th}$:
\begin{multline*}
\Langle\O,\O(1,0)(-\ebe),\O(1,0),\O(2,0)(-2\ebe),\O(2,0)(-\ebe), \\
\O(0,1),\O(1,1)(-\ebe),\O(1,1),\O(2,1)(-2\ebe),\O(2,1)(-\ebe),D(X)\Rangle.
\end{multline*}
Observing that the $3^\mathrm{rd}$ and $4^\mathrm{th}$ terms are orthogonal, and the $8^\mathrm{th}$ and $9^\mathrm{th}$, we swap them to give
\begin{multline*}
\Langle\O,\O(1,0)(-\ebe),\O(2,0)(-2\ebe),\O(1,0),\O(2,0)(-\ebe), \\
\O(0,1),\O(1,1)(-\ebe),\O(2,1)(-2\ebe),\O(1,1),\O(2,1)(-\ebe),D(X)\Rangle.
\end{multline*}
Similarly the $4^\mathrm{th}$ and $5^\mathrm{th}$ terms are orthogonal to the $6^\mathrm{th},\,7^\mathrm{th}$ and $8^\mathrm{th}$, so we move them past and then left mutate $D(X)$ past the 4 terms to its left:
\begin{multline*}
\Langle\O,\O(1,0)(-\ebe),\O(2,0)(-2\ebe),\O(0,1),\O(1,1)(-\ebe), \\
\O(2,1)(-2\ebe),D(X),\O(1,0),\O(2,0)(-\ebe),\O(1,1),\O(2,1)(-\ebe)\Rangle.
\end{multline*}
Finally, we left mutate the 4 terms to the right of $D(X)$ to the front of the exceptional collection, thus tensoring them with $K_\H=\O(-3,-1)(2\ebe)$:
\begin{multline*}
\Langle\O(-2,-1)(2\ebe),\O(-1,-1)(\ebe),\O(-2,0)(2\ebe),\O(-1,0)(\ebe),\O, \\
\O(1,0)(-\ebe),\O(2,0)(-2\ebe),\O(0,1),\O(1,1)(-\ebe),\O(2,1)(-2\ebe),D(X)\Rangle.
\end{multline*}
Using \eqref{rholine} we can write this in terms of sheaves pulled back from $\PP^4\times\PP^1$:
\begin{multline*}
\Langle\rho^*\O(-2,-1),\rho^*\O(-1,-1),\rho^*\O(-2,0),\rho^*\O(-1,0),\O,\\
\rho^*\O(1,0),\rho^*\O(2,0),\rho^*\O(0,1),\rho^*\O(1,1),\rho^*\O(2,1),D(X)\Rangle.
\end{multline*}
We identify this with
\beq{semi3}
\Langle\rho^*D(\PP^4\times\PP^1),D(X)\Rangle
\eeq
by using the standard exceptional collection
\begin{multline*}
D(\PP^4\times\PP^1)=\Langle\O(-2,-1),\O(-1,-1),\O(-2,0),\O(-1,0),\O, \\ \O(1,0),
\O(2,0),\O(0,1),\O(1,1),\O(2,1)\Rangle.
\end{multline*}
Comparing \eqref{semi3} with \eqref{semi2} gives the equivalence $D(X)\cong D(Y)$.
\end{proof}

Nick Addington and Paul Aspinwall pointed out that in this example $X$ and $Y$ are birational. In fact we have the following. Recall the map $\rho\colon P\to\PP^4$ \eqref{p1bdl}.
 
\begin{prop}\label{wayprop}
The compositions
$$
X\Into P\Rt{\rho}\PP^4
$$
and
$$
Y\Into\PP^4\times\PP^1\To\PP^4
$$
project $X$ and $Y$ to the same quintic 3-fold $Q$. Generically $Q$ has 36 ODPs, in which case we obtain $Y$ from $X$ by flopping all 36 exceptional $\PP^1$s.
\end{prop}

\begin{proof}
The key is the isomorphism \eqref{isoref}. Let
\begin{align*}
    uq_0+tq_\infty\in H^0(\O_{\PP^4\times\PP^1}(2,1)) \\
    uc_0+tc_\infty\in H^0(\O_{\PP^4\times\PP^1}(3,1))
\end{align*}
be the corresponding pencils of quadrics and cubics respectively.
Here $u,t$ are the standard sections of $\O_{\PP^1}(1)$ giving the homogeneous coordinates of the point $[u:t]\in\PP^1$.
 
Consider the quintic $Q\subset\PP^4$ defined by the equation
$$
q_0c_\infty-q_\infty c_0=0.
$$
Generically $\{q_0=0=q_\infty=c_0=c_\infty\}$ is $2\cdot2\cdot3\cdot3=36$ reduced points, which are then the ODPs of $Q$.
 
The blow up of $Q$ in the Weil divisor $q_0=0=q_\infty$ (or equivalently the blow up in the Weil divisor $c_0=0=c_\infty$) gives the small resolution
$$
\big\{uq_0+tq_\infty=0=uc_0+tc_\infty\big\}\subset\PP^4\times\PP^1.
$$
But this is precisely the definition of $Y$. Flopping all 36 exceptional curves gives instead the blow up of $Q$ in the Weil divisor $q_0=0=c_0$ (or equivalently in $q_\infty=0=c_\infty$).
This is
\beq{XX}
\big\{Uq_0+Tc_0=0=Uq_\infty+Tc_\infty\big\}\subset P,
\eeq
where $U$ is the section of $\O_\rho(1)\otimes\rho^*\O_{\PP^4}(-2)$ vanishing on the section $\PP(\O_{\PP^4})\subset P=\PP\big(\O_{\PP^4}(1)\oplus\O_{\PP^4}\big)$, and $T$ is the section of $\O_\rho(1)\otimes\rho^*\O_{\PP^4}(-3)$ vanishing on the section $\PP(\O_{\PP^4}(1))\subset P=\PP\big(\O_{\PP^4}(1)\oplus\O_{\PP^4}\big)$.
But \eqref{XX} is precisely the definition of $X$.
\end{proof}

This flop already implies that $X$ and $Y$ have equivalent derived categories, via an equivalence which takes the structure sheaf $\O_x$ of a general point $x\in X$ to the structure sheaf of the corresponding point of $Y$.
Our equivalence, however, can be calculated to take $\O_x$ to a complex of rank $-3$ and Euler characteristic $-137$.
Via the flop equivalence, therefore, we should think of our cubic fourfold constructions as instead giving an exotic derived autoequivalence of $X$ (or $Y$).
\bigskip

\begin{rmk}
As the referee pointed out, there is another way to see the birational equivalence of Proposition \ref{wayprop}.
	View $Q\subset\P^4$ as the degeneracy locus of the map 
	$\phi\colon\O^{\oplus 2}\to\O(2)\oplus\O(3)$ defined by
	\[
	\phi=
		\begin{pmatrix}
			q_0 & q_\infty \\
			c_0 & c_\infty
		\end{pmatrix}.
\]
That is, $\phi$ is invertible on $\PP^4\setminus Q$, has rank $1$ on the smooth locus of $Q$ and is zero on the 36 ODPs of $Q$.
 
The projectivisation of its fibrewise kernel\footnote{ Given a map of vector bundles $\phi\colon E\to F$ over a base $B$ we can define its projectivised kernel $\PP(\ker\phi)\to B$ inside $\PP(E)\rt{\pi}B$ to be the zeros of the corresponding section $\phi\in H^0\big(\PP(E),\pi^*F\otimes\O_\pi(1)\big)$. Replacing $\phi$ by its adjoint $\phi^*$ gives instead the projectivised dual cokernel.}
lies in $\PP(\O^{\oplus2})=\PP^4\times\PP^1$ and is the small resolution $X$. The projectivisation of its fibrewise dual cokernel lies in $\PP\big(\O(-2)\oplus\O(-3)\big)=P$ and gives the small resolution $Y$.
\end{rmk}

\begin{rmk}
The observant reader will have noticed that in each of our examples we have effectively taken homologically dual varieties
$$
A\To\PP(V) \quad\mathrm{and}\quad B\To\PP(V^*)
$$
and restricted attention to a linear subsystem
\beq{subsys}
\PP(W^\perp)\,\subset\ \PP(V^*).
\eeq
Here we have fixed some $W\subset V$, defining subvarieties $\PP(W)\subset\PP(V)$ and (by basechange) $A_{\PP(W)}\subset A$, so that \eqref{subsys} is the linear subsystem of hyperplanes vanishing on them.

This gives new HP dual varieties
$$
\Bl_{A_{\PP(W)}}(A)\To\PP(V/W) \quad\mathrm{and}\quad B_{\PP(W^\perp)}\To\PP(W^\perp),
$$
where the first arrow is induced by the natural projection $\Bl_{\PP(W)}(\PP(V))\To\PP(V/W)$.
Details will appear in \cite{CT}, but we have been using the simplest form of this duality: its application to pencils $\PP(W^\perp)$.

In our examples we took $A = \P^5$, $V = \Sym^3 \C^6$ and $W^\perp \subset \Sym^3(\C^6)^*$ the linear system of cubics vanishing on either a plane or a single point $0\in\PP^5$.
The dual $B$ was in both cases a noncommutative variety which became commutative on basechange to $\P(W^\perp)$.
\end{rmk}

\bibliographystyle{halphanum}
\bibliography{references}

\begin{thebibliography}{HaTu}

\bibitem[Ad1]{Add} N. Addington, \emph{The derived category of the intersection of four quadrics}. \href{http://arxiv.org/abs/0904.1764}{arXiv:0904.1764}.

\bibitem[Ad2]{AddThesis} N. Addington, \emph{Spinor sheaves and complete intersections of quadrics}, Ph.D., University of Wisconsin-Madison, 2009. \href{http://math.duke.edu/~adding/papers/phd_thesis.pdf}{www.math.duke.edu$/^\sim$adding}.

\bibitem[ABB]{ABB} A. Auel, M. Bernardara, M. Bolognesi, \emph{Fibrations in complete intersections of quadrics, {C}lifford algebras, derived categories, and rationality problems}, J. Math. Pures Appl. {\bf 102} (2014), 249--291. \href{http://arxiv.org/abs/1109.6938}{arXiv:1109.6938}.

\bibitem[B$^+$]{BDFIK} M. Ballard, D. Deliu, D. Favero, M. U. Isik, L. Katzarkov, \emph{Homological Projective Duality via Variation of Geometric Invariant Theory Quotients}. \href{http://arxiv.org/abs/1306.3957}{arXiv:1306.3957}.


% \bibitem[BC]{CB} L. Borisov, A.~C\u{a}ld\u{a}raru, \emph{The Pfaffian-Grassmannian derived equivalence}, J. Algebraic Geom. \textbf{18} (2009), 201--222.
% \href{http://arxiv.org/abs/math/0608404}{math.AG/0608404}.

\bibitem[Be]{Be} M. Bernardara, \emph{A semiorthogonal decomposition for Brauer Severi schemes}, Math. Nachr. \textbf{282} (2009), 1406--1413. \href{http://arxiv.org/abs/math/0511497}{math.AG/0511497}.

\bibitem[CT]{CT} F. Carocci and Z. Turcinovic, \emph{Homological projective duality and blowups}, preprint.

\bibitem[HaTu]{HaTu} J. Harris and L. Tu, \emph{On symmetric and skew-symmetric determinantal varieties}, Topology \textbf{23} (1984), 71--84.


\bibitem[Ku1]{V14} A. Kuznetsov, \emph{Derived category of a cubic threefold and the variety $V_{14}$}, Proc. Steklov Inst. Math. \textbf{246} (2004), 171--194. \href{http://arxiv.org/abs/math/0303037}{math.AG/0303037}.

\bibitem[Ku2]{HPD} A. Kuznetsov, \emph{Homological projective duality}, Pub. Math. I.H.E.S. \textbf{105} (2007), 157--220. \href{http://arxiv.org/abs/math/0507292}{math.AG/0507292}.

\bibitem[Ku3]{Cu} A. Kuznetsov, \emph{Derived categories of cubic fourfolds}, Progress in Mathematics \textbf{282} (2010), 219--243. \href{http://arxiv.org/abs/0808.3351}{arXiv:0808.3351}.

\bibitem[Ku4]{Kuf} A. Kuznetsov, \emph{Scheme of lines on a family of 2-dimensional quadrics: geometry and derived category}, Math. Z. \textbf{276} (2014), 655--672. \href{http://arxiv.org/abs/1011.4146}{arXiv:1011.4146}.

\bibitem[Ku5]{kusod} A. Kuznetsov, \emph{Base change for semiorthogonal decompositions}, Compos. Math. {\bf{147}} (2011), 852--876. \href{http://arxiv.org/abs/0711.1734}{arXiv:0711.1734}.

\bibitem[Ku6]{kuquad} A. Kuznetsov, \emph{Derived categories of quadric fibrations and intersections of quadrics}, Adv. Math. {\bf 218} (2008), 1340--1369. \href{http://arxiv.org/abs/math/0510670}{math.AG/0510670}.

\bibitem[Ku7]{KuGr} A. Kuznetsov, \emph{Homological projective duality for Grassmannians of lines}. \href{http://arxiv.org/abs/math/0610957}{math.AG/0610957}.

\bibitem[Or]{Or} D. Orlov, \emph{Projective bundles, monoidal transformations, and derived categories of coherent sheaves}, Izv.
Ross. Akad. Nauk Ser. Mat. \textbf{56} (1992), 852--862; translation in Russian Acad. Sci. Izv. Math. \textbf{41} (1993), 133--141.



\end{thebibliography}

\begin{multicols}{2}

\medskip \noindent {\tt{richard.thomas@imperial.ac.uk}} \\
\noindent Department of Mathematics \\
\noindent Imperial College London\\
\noindent London SW7 2AZ \\
\noindent UK

\medskip
\noindent {\tt{calabrese@rice.edu}} \\
\noindent Rice University MS136 \\
\noindent 6100 Main St. \\
\noindent Houston 77251 TX \\
\noindent USA

\end{multicols}

\end{document}